\documentclass{amsart}
\usepackage{amsmath, amssymb, dsfont}
  \usepackage{paralist}
  \usepackage{graphics} 
  \usepackage{epsfig} 
\usepackage{graphicx}  \usepackage{epstopdf}
 \usepackage[colorlinks=true]{hyperref}
\hypersetup{urlcolor=blue, citecolor=red}

\newtheorem{theorem}{Theorem}[section]
\newtheorem{corollary}{Corollary}

\newtheorem{lemma}[theorem]{Lemma}

\theoremstyle{definition}

\newtheorem{remark}{Remark}

\newcommand{\RR}{\mathbb{R}} 

\renewcommand{\AA}{\mathbb{A}}
\newcommand{\BB}{\mathbb{B}}
\newcommand{\Circle}{\mathbb{S}^{1}} 

\newcommand{\DiffS}{\mathrm{Diff}^{\infty}(\mathbb{S}^{1})} 
\newcommand{\CS}{\mathrm{C}^{\infty}(\mathbb{S}^{1})} 
\newcommand{\GS}{C^{\infty}G} 
\newcommand{\gs}{C^{\infty}\mathfrak{g}} 

\newcommand{\HH}[1]{H^{#1}(\mathbb{S}^{1})}
\newcommand{\Ss}[1]{W^{k,p}(\mathbb{S}^1)}
\newcommand{\norm}[1]{\left\Vert#1\right\Vert}
\newcommand{\abs}[1]{\left\vert#1\right\vert}

\newcommand{\deq}{\stackrel{\mathrm{def}}{=}}

\DeclareMathOperator{\ad}{ad} %

\title[Higher order Camassa-Holm equations on fractional Sobolev spaces ]
{Two-component higher order Camassa-Holm systems with fractional inertia operator: a geometric approach}

\author[Joachim Escher and Tony Lyons]{}


\email{escher@ifam.uni-hannover.de}
\email{tlyons@ucc.ie}

\thanks{TL was supported by the Irish Research Council, Government of Ireland Postdoctoral Fellowship GOIPD/2014/34}

\begin{document}
\maketitle

\centerline{\scshape Joachim Escher}
\medskip
{\footnotesize
 \centerline{Institute for Applied Mathematics}
   \centerline{University of Hanover}
   \centerline{ D-30167 Hanover, Germany}
} 

\medskip

\centerline{\scshape Tony Lyons}
\medskip
{\footnotesize
 \centerline{School of Mathematical Sciences}
   \centerline{University College Cork}
   \centerline{Cork, Ireland}
}

\begin{abstract}
In the following we study  the qualitative properties of solutions
to the geodesic flow induced by a higher order two-component Camassa-Holm system.
In particular, criteria to ensure the existence of temporally global solutions are presented. Moreover
in the metric case, and for inertia operators of order higher than three, the flow is shown
to be geodesically complete.
\end{abstract}

\section{Introduction}\label{Sec1}
In this paper we will investigate a generalised two-component Camassa-Holm equation with fractional order inertia operator, given by the system:
\begin{equation}\label{s1:eq1}
\begin{split}
	&m_t = \alpha u_x - au_xm-um_x - \kappa\rho\rho_x,\quad m = Au\\
	&\rho_t = -u\rho_x - (a-1)u_x\rho,\qquad a\in \RR\setminus \{1\}\\
    &\alpha_t = 0
\end{split}
\end{equation}
where the inertia operator $A=(1-D^2)^s$ belongs to the class of fractional Sobolev norms with $s>1$ and $D=\frac{d}{dx}$, and where $u$ and $\rho$ are defined on the circle $\Circle$ and $\alpha$ is a constant. The meaning of these (dependent) variables will be explained later on. Finally $a$ and $\kappa$ are real parameters. In Section \ref{Sec2} it will be shown that the system in equation \eqref{s1:eq1} corresponds to a metric induced geodesic flow on the infinite dimensional Lie group $\DiffS\circledS \CS\times \RR$, when the parameter $a=2$. We denote by $\DiffS$ the group of orientation preserving diffeomorphisms of the circle, $\CS$ denotes the space of smooth function on $\Circle$ while $\circledS$ denotes an appropriate semi-direct product between the pair. In Section \ref{Sec3} we also present rigorous criteria ensuring the existence of global solutions to the system in equation \eqref{s1:eq1}, for a class of nonlocal inertia operators subject to $s>1$ in the general case $a\neq1$. As a consequence of these general investigations we obtain the following result:
\begin{theorem}[Geodesic completeness]\label{thm:gc}
Let $s>3/2$ and $\kappa\ge 0$ be given and assume that $a=2$. Then the solution to \eqref{s1:eq1} emerging from any initial condition $(u_0,\rho_0)\in \CS\times\CS$ is smooth and exists globally in time, i.e.
\begin{equation*}
    (u,\rho)\in \mathrm{C}^{\infty}([0,\infty),\CS\oplus\CS),
\end{equation*}
meaning that the corresponding flow is geodesically complete.
\end{theorem}

In \cite{EHKL2014} the authors, with collaborators, presented a hydrodynamical derivation of the system \eqref{s1:eq1} with $s=1$ and $\kappa>0$, as a model for water waves with $\alpha$ a constant incorporating an underlying vorticity of the flow, following the works \cite{Iva2009} and \cite{Joh2003}. Additionally it was  shown that the hydrodynamical model obtained when $a=2$ corresponds with the Arnold-Euler equation of a right-invariant metric on the infinite dimensional Lie group $\GS = \DiffS\circledS \CS\times \RR$. Moreover, in line with the work of \cite{EM1970}, local uniqueness and existence of the geodesics on $\GS$ were proven, thus ensuring the well posedness of the system \eqref{s1:eq1} with $s=1$. Furthermore a priori estimates were obtained which established criteria for the global existence of solutions for the hydrodynamical system when $a\neq1$.

The system in \eqref{s1:eq1} incorporates a number of remarkable partial differential equations arising as approximate hydrodynamical models. In the case of irrotational flow ($\alpha = 0$), with $\rho\equiv0$ and $s=1$, the system above reduces to a family of equations parameterised by $a\neq1$, the so-called $b$-equations. This class of $b$-equations share several common structural features investigated in \cite{EY2008, Hen2008, Hen2009b} among other works. There are two scenarios in which the $b$-equations yield integrable models, specifically when $a=2$ corresponding to the Camassa-Holm equation \cite{CHH1994,CH1993}, and $a=3$ giving the Degasperis-Procesi equation \cite{DP1999}. Both systems arise as hydrodynamical models of shallow water waves c.f. \cite{Con2011,Joh2002}, possessing both global solutions and solutions which display wave-breaking in finite time, c.f. \cite{Con2000, CE1998b, CE1998a}.

All of these hydrodynamical models have a geometrical interpretation in terms of a geodesic flow on an appropriate infinite dimensional Lie group. The seminal work of Arnold \cite{Arn1966} reformulated the Euler equation describing an ideal fluid, as a geodesic flow on the group of volume preserving diffeomorphisms of the fluid domain. Following this, Ebin and Marsden reinterpreted this group of volume preserving diffeomorphisms as an inverse limit of Hilbert manifolds, c.f. \cite{EM1970}. The crucial aspect of this approach is that it allows one to reformulate the hydrodynamical model governed by a PDE, namely the Euler equation, in terms of the geodesic equation on the Hilbert manifolds, which is an ODE. The question of well-posedness of the Euler equation may then be reformulated in terms of the local existence and uniqueness of solutions to the geodesic equation.
In this manner well-posedness of the system in \eqref{s1:eq1} was established for the specific case $s=1$ and $a\neq1$ in the work \cite{EHKL2014}, and these results may be extended to ensure the well-posedness of the general system when $s>1$ and $a\in\RR\setminus\{1\}$.

In Section \ref{Sec2} of this paper we demonstrate that the system in \eqref{s1:eq1} can be recast as the flow of a right-invariant metric on the Fr\'{e}chet Lie group $\GS$, when $a=2$. This result is shown to be true for a general class of nonlocal inertia operators of the form $A=(1-D^2)^s$ with $s>1$. The well-posedness of the system remains valid for this general class and the reader is referred to the results presented in \cite{EHKL2014} \S 5, and also \cite{EK2014a} for further relevant discussions. Lastly in Section \ref{Sec3} we establish rigourous criteria ensuring global solutions for the system \eqref{s1:eq1} when $A$ induces a fractional Sobolev norm, for general $a\neq1$. In this regard, the pseudo-differential form of the inertia operator generates additional terms requiring careful analysis of the \textit{a priori} estimates necessary to ensure the existence of global solutions. We close our study by giving a proof of Theorem \ref{thm:gc}.

\section{Geometric Formulation}\label{Sec2}
The work undertaken in \cite{CK2002} investigated the  geodesic equation on $\DiffS$ associated with the right-invariant $H^{1}$ metric, namely the Camassa-Holm equation, which is a reduction of the system being currently being investigated. Meanwhile, in \cite{Esc2012} the short-wave limit of \eqref{s1:eq1}, namely $m=-u_{xx}$ and $s=1$, is shown to correspond to a geodesic flow on a symmetric right-invariant connection on a suitable semi-direct product, when $\alpha = 0$. Moreover it was shown in the case $a=2$ that this connection is associated with a metric induced by the norm
\[\norm{u_{xx}}_{L^2}+\norm{\rho}_{L^2},\]
with $u\in\CS/\RR$ and $\rho\in\CS$. It was shown in \cite{EHKL2014} that the system \eqref{s1:eq1} could be recast as a geodesic flow of a right invariant metric on a suitable semi-direct product, when $s=1$ and $a=2$. In Theorem \ref{s1:thm1} it will be shown that these results extend to the case when $A$ belongs to a general class of nonlocal inertia operators.

Given a Lie group $G$ with corresponding Lie algebra $\mathfrak{g}$, then any inner product on $\mathfrak{g}$ may be extended to a \textit{right-invariant metric} in $G$ by right translation of this inner-product. We may represent the inner product on the Lie algebra by an invertible operator
\[\mathbb{A}:\mathfrak{g}\to\mathfrak{g}^{*}.\]
Then a path $g(t)\subset G$ is a geodesic of the $\mathbb{A}$-induced right-invariant metric of $G$ if and only if the Eulerian velocity of $g(t),$ namely
\[ u(t) := TR_{g(t)^{-1}}\dot{g}(t)\]
satisfies the Arnold-Euler equation given by
\begin{equation*}
\begin{split}
    u_t = -\BB(u,u),\qquad  \BB(u_1,u_2) \deq \frac{1}{2}\left(\ad^{\top}_{u_1}u_2+ \ad^{\top}_{u_2}u_1\right)
\end{split}
\end{equation*}
where $u_{i}\in\mathfrak{g}$ for $i=1,2$ and $\ad_{u}^{\top}$ is the adjoint of $\ad_u$ with respect to $\AA.$ That is to say, given $\Lambda\in \mathcal{L}\left(\mathfrak{g},\mathfrak{g}\right),$ we define the adjoint $\Lambda^{\top}$ with respect to the norm $\langle\cdot,\cdot\rangle$ as follows
\begin{equation*}
    \langle u_1,\Lambda^{\top} u_2\rangle = \langle \Lambda u_1,u_2\rangle, \qquad \langle u_1,u_2\rangle = \int_{\Circle} u_1\AA u_2.
\end{equation*}
In what follows it will be demonstrated that the general system \eqref{s1:eq1} with $a=2$ and $s>1$ may be reformulated as the geodesic flow of a right invariant metric on the infinite dimensional Lie group $\GS:=\DiffS\circledS\CS\times\RR$.

Moreover we denote by $H^s(\Circle)$, where $s\ge 0$, the usual Sobolev spaces over $\Circle$, equipped with the norm
\begin{equation*}
\Vert u\Vert_{H^s}:=\left(\sum_{k\in\mathbb{Z}}(1+\vert k\vert^2)^s\hat u_k^2\right)^{1/2},\qquad u\in L_2(\Circle),
\end{equation*}
where $\hat u_k:= (u\vert \mathrm{\mathbf{e}}_k)_{L_2}$ is the $k$-th Fourier coefficient with respect to the standard
\footnote{For simplicity we normalise the length of the circle $\Circle$ to the value $1$, so that $\mathrm{\mathbf{e}}_k=\exp{ikx}$.}
orthonormal basis $\{\mathrm{\mathbf{e}}_k\,;\, k\in\mathbb{Z}\}$ of $L_2(\Circle)$. We observe that
\begin{equation}\label{eq:sob}
\Vert u\Vert_{H^s}=\int_{\Circle} u Au\,dx.
\end{equation}
Furthermore we note that with $\mathds{1}:=\mathrm{\mathbf{e}}_0$ we have
\begin{equation}\label{eq:A1}
A\mathds{1}=\sum_{k\in\mathbb{Z}}(1+\vert k\vert^2)^s(\mathds{1}\vert \mathrm{\mathbf{e}}_k)_{L_2}\mathrm{\mathbf{e}}_k=(\mathds{1}\vert \mathrm{\mathbf{e}}_0)_{L_2}\mathrm{\mathbf{e}}_0=\mathds{1}.
\end{equation}
The latter observation and the fact that $A$ is self-adjoint particularly implies that
\begin{equation}\label{eq:Ainv}
\int_{\Circle} A^{-1} w\,dx=(A^{-1} w\vert A\mathds{1})_{L_2}=(w\vert\mathds{1})_{L_2}=\int_{\Circle} w\,dx,\quad w\in L_2(\Circle).
\end{equation}
After the above preparations we can state the following result:

\begin{theorem}\label{s1:thm1}
Given $\kappa > 0$ and $a=2$ it is possible to write the system (\ref{s1:eq1}) as the Arnold-Euler equation on the Lie algebra $\gs$ of the Fr\'{e}chet Lie group $\GS$ with an associated inner product
\begin{multline}\label{s2:eq1}
    \langle(u_1,\rho_1,\alpha_1),(u_2,\rho_2,\alpha_2)\rangle = \int_{\Circle}u_1Au_2dx  + \kappa\int_{\Circle}\rho_1\rho_2dx \\ -\frac{1}{2}\int_{\Circle}\left[\alpha_2u_1+\alpha_1u_2\right]dx + \frac{1}{2}\alpha_1\alpha_2
\end{multline}
with $u_i,\rho_i\in\CS$ and $\alpha_i\in\RR$ for $i=1,2$.
\end{theorem}

\begin{proof}
  We may write the inner product in (\ref{s2:eq1}) as
\begin{equation}\label{s2:eq2}
    \langle(u_1\,\rho_1,\alpha_1),(u_2,\rho_2,\alpha_2)\rangle = \int_{\Circle}(u_1,\rho_1,\alpha_1)\cdot\mathbb{A}(u_2,\rho_1,\alpha_2)dx,
\end{equation}
where $\cdot$ is the usual dot product in $\RR^3$. The inertia operator $\AA$ is given by
\begin{equation}\label{s2:eq3}
    \AA(u,\rho,\alpha) = \left(Au - \frac{\alpha}{2},\kappa\rho,\frac{1}{2}\Big(\alpha-\int_{\Circle}u\,dx\Big)\right)
\end{equation}
With $U_i = (u_i,\rho_i,\alpha_i)\in \gs$, where $i=1,2,$ we have
\begin{equation}\label{s2:eq4}
    \ad_{U_1}U_2 = \left(u_{1,x}u_2-u_1u_{2,x},\rho_{1,x}u_2 -\rho_{2,x}u_1,0\right)
\end{equation}
and so upon returning (\ref{s2:eq3}) and (\ref{s2:eq4})  to (\ref{s2:eq2}) we find
\begin{align*}
    \langle\ad_{U_1}U_2,U_3\rangle =& \int_{\Circle}\left(u_{1,x}u_2-u_1u_{2,x},\rho_{1,x}u_2-\rho_{2,x}u_1, 0\right)\cdot\AA(u_3,\rho_3,\alpha_3)\\
                                   =& \int_{\Circle}(u_{1,x}u_2-u_1u_{2,x})Au_3 dx +\kappa\int_{\Circle}\left(\rho_{1,x}u_2-\rho_{2,x}u_1, \right)\rho_3dx\\
                                   &-\frac{\alpha_3}{2}\int_{\Circle}(u_{1,x}u_2-u_1u_{2,x})dx
\end{align*}
Integration by parts along with $[A,D]= 0$ yields
\begin{equation}\label{s2:eq5}
\begin{split}
    \langle\ad_{U_1}U_2,U_3\rangle = & \int_{\Circle}u_2\left(u_{1,x}Au_3 +(u_1Au_{3})_x+ \kappa\rho_{1,x}\rho_3-\alpha_3u_{1,x}\right)\\
    &+ \kappa\int_{\Circle}\rho_2(u_1\rho_3)_x \\
    \deq&\int_{\Circle}(u_2,\rho_2,\alpha_2)\cdot(f,g,0)\\
    \deq & \int_{\Circle}(u_2,\rho_2,\alpha_2)\cdot\AA(\tilde{u},\tilde{\rho},\tilde{\alpha}).
\end{split}
\end{equation}
which defines $\tilde{U} = (\tilde{u},\tilde{\rho},\tilde{\alpha})$ in terms of $U_1$ and $U_3$. Thus we are looking for the unique solution of
\begin{equation*}
A\tilde{u} - \frac{\tilde{\alpha}}{2} = f\qquad \kappa\tilde{\rho} = g\qquad \tilde{\alpha} - \int_{\Circle}\tilde{u}dx = 0,
\end{equation*}
which, invoking \eqref{eq:A1}, is found to be
\begin{equation}\label{s2:eq5.1}
    \tilde{u} = A^{-1}f + \int_{\Circle}fdx\qquad \tilde{\rho} = \frac{1}{\kappa}g\qquad \tilde{\alpha} = 2\int_{\Circle}fdx,
\end{equation}
However by definition we also have
\[\langle\ad_{U_1}U_2,U_3\rangle = \langle U_2,\ad_{U_1}^{\top}U_3\rangle = \int_{\Circle}(u_2,\rho_2,\alpha_2)\cdot\AA(\tilde{u},\tilde{\rho},\tilde{\alpha})\]
in which case it follows that $\ad_{U_1}^{\top}U_{3} = (\tilde{u},\tilde{\rho},\tilde{\alpha})$.
Explicitly we have
\begin{equation}\label{s2:eq6}
\begin{split}
    \tilde{u} =& A^{-1}\left(u_{1,x}Au_3 +(u_1Au_{3})_x+ \kappa\rho_{1,x}\rho_3+\alpha_3u_{1,x}\right)\\
               &+\int_{\Circle} \left(u_{1,x}Au_3 + \kappa\rho_{1,x}\rho_3\right)dx\\
    \tilde{\rho} =& (u_1\rho_3)_x\\
    \tilde{\alpha} =&  -2\int_{\Circle}\left(u_{1,x}Au_3 + \kappa\rho_{1,x}\rho_3\right)dx
\end{split}
\end{equation}
Lastly we note by the symmetry of $A$ along with its commutativity with $D$, that
\begin{equation}\label{s2:eq7}
\int_{\Circle}u_{1,x}Au_3 = -\int_{\Circle}u_1Au_{3,x} = -\int_{\Circle}Au_1u_{3,x}
\end{equation}
and so the integral is antisymmetric in $u_1$ and $u_3.$
Hence substituting  $U=(u,\rho,\alpha)$ in (\ref{s2:eq6}) and employing (\ref{s2:eq7}) yields
\begin{equation}\label{s2:eq8}
\begin{split}
 \BB(U,U) =& \ad_{U}^{\top}U
          = \left(A^{-1}\left[u_xAu + (uAu)_x -\alpha u_x + \kappa\rho\rho_x\right], (u\rho)_x,0\right).
\end{split}
\end{equation}
Consequently the Arnold-Euler equation $U_{t} = -\BB(U,U)$
corresponds precisely with the system (\ref{s1:eq1}) when $a=2.$
\end{proof}
\begin{remark}
Note that the bilinear form given by \eqref{s2:eq1} is positive definite for any choice of $\kappa\ge 0$ and $\alpha\in\mathbb{R}$.
\end{remark}
\begin{proof} Define
\begin{equation*}
\Vert (u,\rho,\alpha)\Vert_{\mathbb{A}}^2:=\int_{\Circle}u Au\,dx+\kappa\int_{\Circle}\rho^2\,dx-\alpha\int_{\Circle}u\,dx+\frac{\alpha^2}{2}
\end{equation*}
for $u,\;\rho \in \CS$ and $\alpha\in\mathbb{R}$. Invoking $\int u^2\le \Vert u\Vert_{H^s}=\int uAu$, we get
$$
\Vert (u,\rho,\alpha)\Vert^2_{\mathbb{A}}\ge \int_{\Circle}(u-\frac{\alpha}{2})^2dx+\kappa\int_{\Circle}\rho^2dx.
$$
Thus $\Vert (u,\rho,\alpha)\Vert^2_{\mathbb{A}}>0$ for all $(u,\rho,\alpha)\in \gs$ with $u\ne\alpha/2$ or $\rho\ne 0$. But
$$
\Vert (\frac{\alpha}{2},0,\alpha)\Vert^2_{\mathbb{A}}=\frac{\alpha^2}{4},
$$
which completes the proof.
\end{proof}
\begin{remark} Given $u\,\;\rho\in\CS$ and $\alpha\in\mathbb{R}$, we have that
\begin{equation}\label{eq:aprio}
\frac{3}{4}\Vert u\Vert_{H^s}^2+\kappa\Vert\rho\Vert_{L_2}^2\le\Vert(u,\rho,\alpha)\Vert^2_{\mathbb{A}} +
\frac{\alpha^2}{2}.
\end{equation}
\end{remark}
\begin{proof}
Using the elementary inequality $-u\alpha\ge -u^2/4 - \alpha^2$, we get
$$
-\int_{\Circle}u\alpha\, dx\ge - \frac{1}{4}\int_{\Circle}u^2 dx-\alpha^2\ge - \frac{1}{4}\Vert u\Vert_{H^s}^2-\alpha^2
$$
which readily gives the assertion.
\end{proof}
\begin{remark}\label{rm:3}
Consider the case $a=2$ and let $(u,\rho)$ be a solution to \eqref{s1:eq1} on the time interval $J$. Then the quantity
$
\Vert (u,\rho,\alpha)(t)\Vert_{\mathbb{A}}
$
is constant on $J$.
\end{remark}
\begin{proof}
We have
\begin{equation}\label{eq:tvar}
\frac{1}{2}\frac{d}{dt} \Vert (u,\rho,\alpha)(t)\Vert_{\mathbb{A}}^2=\int_{\Circle}u_tAu\,dx+\kappa \int_{\Circle} \rho_t\rho\,dx - \alpha\int_{\Circle} u_t\,dx.
\end{equation}
Invoking \eqref{s1:eq1} we find because of $[A,D]=0$:
\begin{align*}
    \int_{\Circle} u_t Au =& \int_{\Circle}m_t u\\
                                   =&\ \alpha\int_{\Circle}uu_x- 2\int_{\Circle}u u_x Au-\int_{\Circle}u^2 Au_x-\kappa\int_                                      {\Circle}\rho\rho_x\\
                                   =&\ \frac{\alpha}{2}\int_{\Circle}(u^2)_x- \int_{\Circle}(u^2)_xAu +
									  \int_{\Circle}(u^2)_xAu
										-\kappa\int_{\Circle}\rho\rho_x u\\
								   =&\ -\kappa\int_{\Circle}\rho\rho_x u.
\end{align*}
For the second integral of the right-hand side of \eqref{eq:tvar} we get
$$
\kappa\int_{\Circle}\rho_t\rho =-  \kappa\int_{\Circle}(\rho u)_x\rho =  \kappa\int_{\Circle}\rho\rho_x u.
$$
Thus $\int u_t Au+\kappa\int\rho_t\rho=0$ and we are left to show that $\int u_t=0$. Using \eqref{eq:Ainv} we obtain
\begin{align*}
    \int_{\Circle} u_t  =&\ \alpha\int_{\Circle} A^{-1}u_x - 2\int_{\Circle}A^{-1}(u_x Au)-
							\int_{\Circle}A^{-1}(uAu_x)-\kappa\int_{\Circle}\rho\rho_x\\
                        =&\ \alpha\int_{\Circle} u_x - 2\int_{\Circle}u_x Au-
							\int_{\Circle}uAu_x-\kappa\int_{\Circle}(\rho^2)_x.
\end{align*}
Recalling $[A,D]=0$, integration by parts yields $\int u_t=\int uAu_x$. But $\int u Au_x=0$ by \eqref{s2:eq7} and the proof is complete.
\end{proof}
Following the work of Arnold \cite{Arn1966}, Ebin \& Marsden \cite{EM1970} considered the group of smooth diffeomorphisms as an inverse limit of Hilbert manifolds.  This essentially allows one to recast classes of PDEs, namely the Euler equation, in terms of an ODE, that is to say the geodesic equation on these Hilbert manifolds. These techniques were applied in \cite{Mis2002,CK2003} for the periodic Camassa-Holm equation, while in \cite{EK2011} the Degasperis-Procesi equation was recast as a non-metric geodesic flow. Moreover in \cite{CK2003a, EK2014} these methods were extended to right-invariant metrics induced by fractional Sobolev norms, following from \cite{Kol2009}. In \cite{EHKL2014} the same techniques were applied to the hydrodynamical reduction of \eqref{s1:eq1}, namely $s=1$,  thereby establishing well-posedness of that system. Having reformulated the system \eqref{s1:eq1} as a right-invariant metric flow on the Fr\'{e}chet Lie group $\DiffS\circledS\CS\times\RR$ when $a=2$ and $s>1$, these same techniques also ensure the well-posedness of the system for this general class of fractional Sobolev norms with $a\in\RR\setminus\{1\}$.

\section{Global Solutions}\label{Sec3}
As was pointed out in Section \ref{Sec1}, the Camassa-Holm equation arises as a reduction of the system \eqref{s1:eq1}, namely when $\rho\equiv0$, $a=2$, $\alpha=0$ and $s=1$, and given the classical results of \cite{Con2000, CE1998a,CE1998b} it is clear that not all solutions of either system should exist globally. However that is not to say that global solutions for the class of systems currently being investigated  do not exist. A reduction of the system to the two-component Camassa-Holm equation, namely $s=1$, $\alpha =0$ and $a=2$, was found to possess several notable features in the works \cite{CI2008, Hen2009}, including global solutions, thus providing motivation for the investigation of global existence criteria for the general system \eqref{s1:eq1}.

The solutions $(u,\rho)$ to the system in equation \eqref{s1:eq1}, with associated initial data $(u_0,\rho_0)$, correspond to smooth paths in
\begin{equation*}
    (u,\rho)\in \mathrm{C}^{\infty}(J,\CS\oplus\CS),
\end{equation*}
where $J$ is the maximal time interval of their existence. In what follows we present a collection of \textit{a priori} estimates which establish criteria for the existence of global solutions for \eqref{s1:eq1} when $A$ induces a class of fractional Sobolev norms.

\begin{lemma}\label{s3:lem1} We have
\begin{equation}\label{s3:eq1}
	(\rho\circ\phi)(t)\cdot\phi_x^{a-1}(t) = (\rho\circ\phi)(0)\cdot\phi_x^{a-1}(0)\quad \forall t\in J.
\end{equation}
\end{lemma}

\begin{proof} Let $\phi$ be the flow of the vector field $u$. Then we have
\begin{equation}\label{s3:eq1a}
\phi_t = u\circ\phi
\end{equation}
and so it follows
\begin{equation}\label{s3:eq2}
\begin{split}
	\frac{d}{{d}t}(\rho\circ\phi\cdot\phi_x^{a-1}) &=\left((\rho_t+u_x\rho+(a-1)\rho u_x)\circ\phi\right)\cdot\phi^{a-1}_x=0,
\end{split}
\end{equation}
since $(u,\rho)$ solves the system (\ref{s1:eq1}).
\end{proof}

\begin{corollary}\label{s3:cor1}
If $\rho_0 > 0$, it follows that $\rho(t) > 0$ for all $t\in J$.
\end{corollary}

\begin{proof}
It follows from Lemma \ref{s3:lem1} that
\begin{equation}\label{s3:eq3}
\begin{split}
    (\rho\circ\phi)(t)\cdot\phi_x^{a-1}(t) &= (\rho\circ\phi)(0)\cdot\phi_x^{a-1}(0)\\
     (\rho\circ\phi)(t) &=  (\rho\circ\phi)(0)\cdot\frac{\phi_x^{a-1}(0)}{\phi_x^{a-1}(t)}
\end{split}
\end{equation}
However since $\phi(t)$ is an orientation preserving diffeomorphism we must have
\[\frac{\phi_x(0)}{\phi_x(t)} > 0\quad \forall\ t\in J,\]
in which case if $\rho(0)>0$ it follows from (\ref{s3:eq3}) that
\begin{equation}\label{s3:eq4}
(\rho\circ\phi)(t) >0
\end{equation}
for all $t\in J.$
\end{proof}

\begin{corollary}\label{s3:cor2}
 If $\norm{u_x(t)}_\infty$ is bounded on any bounded subinterval of $J$ then $\norm{\rho(t)}_\infty$ is also bounded on any bounded subinterval of $J.$
\end{corollary}

\begin{proof}
Define $\phi$ as the flow of  the time dependent vector field $u$ and set
\[\gamma(t) = \max_{x\in\Circle}\frac{1}{\phi_x(t)}\]
where we also note $\gamma(t)$ is continuous. Let $I \subset J$ where $I$ is bounded and denote
\begin{equation}\label{K}
K = \sup_{t\in I}\norm{u_x(t)}_\infty.\end{equation}
Given (\ref{s3:eq1a}) we see that
\begin{equation}\label{s3:eq5}
    \left(\frac{1}{\phi_x}\right)_t = -\frac{\phi_{tx}}{\phi_x^2}= -\frac{(u\circ\phi)_x}{\phi_x^2} = -\frac{u_x\circ\phi}{\phi_x}
\end{equation}
which following an integration, and using the definition of $\gamma$ and the bound \eqref{K} leads us to the inequality
\begin{equation}\label{s3:eq7}
    \gamma(t) - \gamma(0) \leq K \int_{0}^{t}\gamma(s)\rm{d} s.
\end{equation}
A subsequent application of Gronwall's Lemma ensures
\begin{equation}\label{s3:eq8.0}
    \gamma(t)\leq\gamma(0)e^{Kt},
\end{equation}
that is to say $\frac{1}{\phi_x(t)}$ is bounded on $I.$
It follows from Lemma \ref{s3:lem1} that $\norm{\rho(t)}_{\infty}$ is bounded on $I$.
\end{proof}

\subsection{Estimates for fractional Sobolev norms}\label{Sec3.1}
When we generalise the inertia operator according to $A = (1-D^2)^{s}$ and $s\in(1,\infty)$, the quantity $u_xm$ can no longer be written as a total derivative and so the \textit{a priori} estimates obtained in \cite{EHKL2014} \S6, must be generalised to account for this term.
The operator $\Lambda^{s}=(1-D^2)^{\frac{s}{2}}$ is a topological isomorphism of $H^{s}(\Circle)$ onto $L^2(\Circle)$.
Providing $H^{s}(\Circle)$ with a Hilbert space structure makes $\Lambda^{-s}(\Circle)$ an isometric isomorphism of $L^2(\Circle)$ onto $H^s(\Circle)$.
That is to say, if $(u,v)\in H^s(\Circle)\times H^s(\Circle),$ their inner product may be written as
\begin{equation}\label{s3:eq8}
    \langle u,v\rangle_{H^s} = \langle \Lambda^{s} u,\Lambda^{s} v\rangle_{L^2}.
\end{equation}
Moreover since $\Lambda^{s} \in \mathcal{L}\left(H^s,L^2\right)$ it follows
\begin{equation}\label{s3:eq9}
    \norm{\Lambda^{s}u}_{L^2}\lesssim\norm{u}_{H^s}.
\end{equation}
The Sobolev imbedding $H^s(\Circle)\hookrightarrow L_\infty(\Circle)$ for $s>1/2$ implies the following bound
\begin{equation}\label{e3:eq10}
    \norm{u}_\infty\lesssim \norm{u}_{H^s} \qquad \forall u\in H^s(\Circle),
\end{equation}
while the fact $D\in\mathcal{L}\left(H^{s+1}(\Circle),H^{s}(\Circle)\right)$ ensures $\norm{D u}_{H^s} \leq C\norm{u}_{H^{1+s}}$
and consequently we have the following estimate for $u_x$
\begin{equation}\label{s3:eq11}
    \norm{u_x}_\infty\lesssim \norm{u}_{H^s},\qquad\text{if}\ s>\frac{3}{2}.
\end{equation}
Finally we note that
\begin{equation}\label{u - m}
\Vert u_x\Vert_{L_\infty}\le \Vert u\Vert_{H^{2s}}=\Vert Au\Vert_{L_2}=\Vert m\Vert_{L_2},\quad u\in \CS,
\end{equation}
by the Sobolev embedding theorem based on the assumption $s>3/2$.

\noindent
In addition to the aforementioned estimate for $\norm{u_x}_{\infty}$, we have the related result bounding $\norm{m}_{H^{k}}+\norm{\rho}_{H^{k+1}}$ given by the following:
\begin{lemma}\label{s3:lem2}
Suppose $\norm{u_{x}}_{\infty}$ is bounded on any bounded subinterval of $J.$ Then
\[\norm{m}_{H^k}^2 + \norm{\rho}_{H^{k+1}}^2\]
is also bounded, on any bounded subinterval of $J$ for all integers $k\geq0$.
\end{lemma}

\begin{proof}The proof begins by demonstrating that the quantity $\norm{m}_{L^2}+\norm{\rho}_{H^1}$ remains bounded when $\norm{u_x}_{\infty}$ remains bounded on a bounded subset of $J$.
\begin{equation}\label{s4:eq12}
\begin{split}
\frac{d}{dt} \left(\norm{m}_{L^2}^2 + \norm{\rho}_{H^1}^2\right)&=\int_{\Circle} mm_tdx + \int_{\Circle}(\rho\rho_t+\rho_x\rho_{xt})\\
                                                                &= \int_{\Circle} m(\alpha u_x - au_xm-um_x -\kappa\rho\rho_x)\\
                                                                &\quad+\int_{\Circle}\rho_x(-u\rho_{xx} - au_x\rho_x -(a-1)\rho u_{xx} )\\
                                                                &\qquad +\int_{\Circle}\rho(-u\rho_x - (a-1)u_x\rho) \\
\end{split}
\end{equation}
Integration by parts yields
\[-\int_{\Circle}umm_xdx= \frac{1}{2}\int_{\Circle}u_xm^2\qquad -\int_{\Circle}u\rho\rho_{x}dx=\frac{1}{2}\int_{\Circle}u_x\rho^2dx\]
which in conjunction with the Cauchy-Schwarz inequality provides the following estimate
\begin{equation}\label{s3:eq13}
\frac{d}{dt}\left(\norm{m}_{L^2}+\norm{\rho}_{H^2}\right)
\lesssim\int_{\Circle}\alpha u_{x}mdx+
\left(1+\norm{u_x}_{\infty}+\norm{\rho}_{\infty}\right)\left(\norm{m}_{L^2}^2+\norm{\rho}_{H^{1}}^2\right).
\end{equation}
Meanwhile, H\"{o}lder's inequality applied to the remaining term yields
\begin{equation*}
\int_{\Circle}\alpha u_xm\leq\abs{\alpha}\norm{u_x}_{\infty}\int_{\Circle}\abs{m}dx\lesssim\norm{u_x}_{\infty}\norm{m}_{L^2}\\
\end{equation*}
and it follows at once that
\begin{equation}\label{s3:eq14}
    \int_{\Circle}\alpha u_{x}mdx\lesssim \norm{u_x}_{\infty}^2 + \norm{m}_{L^2}^2.
\end{equation}
Upon combining these estimates we find
\begin{equation}\label{s3:eq15}
\frac{d}{dt}\left(\norm{m}_{L^2}^2 +\norm{\rho}_{H^1}^2\right)\lesssim \norm{u_x}_{\infty}^2+ \left(1+\norm{u_x}_\infty +\norm{\rho}_\infty\right)\left(\norm{m}_{L^2}^2+\norm{\rho}_{H^1}^2\right).
\end{equation}
A straightforward application of the integral form of Gronwall's inequality ensures that the quantity $\norm{m}_{L^2}^2 +\norm{\rho}_{H^1}^2$ is bounded for all $t \in J$ when $\norm{u_x}_{\infty}$ is bounded in $J$.
\begin{remark}
Gronwall's Lemma requires that the functions $m$ and $\rho$ be continuous, however these functions are in fact smooth and so an application of the lemma is permitted.
\end{remark}

At the next iteration, we wish to show the quantity $\norm{m}_{H^1}^2 + \norm{\rho}_{H^{2}}^2$
is also bounded under the same conditions. Specifically we want to show
\begin{equation}\label{s3:eq16}
\frac{d}{dt}\left(\norm{m}_{H^1}^2+\norm{\rho}_{H^{2}}^2\right)\lesssim\norm{m}_{L^2}^2+\left(1+\norm{m}_{L^2}
+\norm{\rho}_{H^{1}}\right)\left(\norm{m}_{H^1}^2+\norm{\rho}_{H^{2}}^2\right).
\end{equation}
We already have
\begin{equation*}
\begin{split}
\frac{d}{dt}\left(\norm{m}_{L^2}^2+\norm{\rho}_{H^1}^2\right)&\lesssim\norm{u_x}_{\infty}^2+\left(1+\norm{u_x}_\infty +\norm{\rho}_\infty\right)\left(\norm{m}_{L^2}^2 +\norm{\rho}_{H^1}^2\right)\\
&\lesssim\norm{m}_{L^2}^2+\left(1+\norm{m}_{L^2}+\norm{\rho}_{H^{1}}\right)\left(\norm{m}_{H^1}^2+\norm{\rho}_{H^{2}}^2\right),
\end{split}
\end{equation*}
with the second relation being a consequence of the imbeddings $H^2\hookrightarrow H^1$ and $H^1\hookrightarrow L_\infty,$
and the estimate \eqref{u - m}. As such we need only show
\begin{equation}\label{s3:eq17}
\frac{d}{dt}\left(\norm{m_x}_{L^2}^2+\norm{\rho_{xx}}_{L^2}^2\right)\lesssim\norm{m}_{L^2}^2+\left(1+\norm{m}_{L^2} +\norm{\rho}_{H^{1}}\right)\left(\norm{m}_{H^1}^2+\norm{\rho}_{H^{2}}^2\right).
\end{equation}
Considering the first term of \eqref{s3:eq17} we have
\begin{equation}\label{s3:eq18}
\begin{split}
\frac{d}{dt}\norm{m_x}_{L^2}^2
	&=\int_{\Circle} m_x(\alpha u_{xx}-au_{xx}m-au_xm_x - u_xm_x - um_{xx} - \kappa\rho_x^2 - \kappa\rho\rho_{xx})dx.\\
\end{split}
\end{equation}
Integration by parts yields $-\int_{\Circle} um_xm_{xx}dx = \frac{1}{2}\int_{\Circle} u_xm_x^2dx$; thus the Cauchy-Schwarz inequality provides a bound for the norm $\norm{m_x}_{L^2}$ according to
\begin{equation}\label{s3:eq19}
\begin{split}
\frac{d}{dt}\norm{m_x}^2 &\lesssim \norm{u}_{H^2}\norm{m}_{H^1} + \norm{u}_{H^2}\norm{m}_{H^1}^2 + \norm{u_x}_\infty\norm{m}_{H^1}^2 \\
&\quad + \norm{m}_{H^1}\norm{\rho_x^2}_{L^2} + \norm{\rho}_\infty\norm{\rho}_{H^2}\norm{m}_{H^1},
\end{split}
\end{equation}
where we have  used the fact that $\HH{q}$ is a multiplicative algebra for $q>\frac{1}{2}$ thus ensuring
\[\int_{\Circle} u_{xx}mm_xdx \lesssim \norm{u}_{H^2}\norm{m}_{H^1}^2.\]
Moreover in line with the previous iteration, the Cauchy-Schwarz inequality ensures
\begin{equation}
    \int_{\Circle}\alpha u_{xx}m_xdx\lesssim\norm{u}_{H^2}\norm{m}_{H^1} \lesssim \norm{u}_{H^2}^2+\norm{m}_{H^1}^2 .
\end{equation}
Thus upon implementing the estimates
\begin{equation*}
\norm{u}_{H^2}\lesssim \norm{m}_{L^2},\quad \norm{u_x}_{\infty} \lesssim \norm{m}_{L^2},\quad
\norm{\rho}_{\infty}\lesssim\norm{\rho}_{H^1},\quad \norm{\rho_x^2}_{L^2} \lesssim\norm{\rho}_{H^1}\norm{\rho}_{H^2}
\end{equation*}
we obtain
\begin{equation}\label{s3:eq20}
\frac{d}{dt}\norm{m_x}^2 \lesssim \norm{m}_{L^2}^2+\left(1 + \norm{m}_{L^2}+\norm{\rho}_{H^1}\right)\left(\norm{m}_{H^1}^2 + \norm{\rho}_{H^2}^2\right).
\end{equation}
Meanwhile the term $\frac{d}{dt}\norm{\rho_{xx}}_{L^2}^2$ may be written according to
\begin{multline}\label{s3:eq21}
\frac{d}{dt}\norm{\rho_{xx}}_{L^2}^2 = \int_{\Circle}\rho_{xx}(-u\rho_x-(a-1)u_x\rho)_{xx}dx\\
= \left(a-\frac{1}{2}\right)\int_{\Circle}\left(u_{xxx}\rho_x^2-u_x\rho_{xx}^2\right)dx + (a-1)\int_{\Circle}u_{xxx}\rho\rho_{xx}dx.
\end{multline}
The Cauchy-Schwarz inequality in conjunction with the estimate $\norm{u}_{H^3}\lesssim\norm{m}_{H^1}$ yields
\begin{equation}\label{s3:eq22}
\begin{split}
\frac{d}{dt}\norm{\rho_{xx}}_{L^2}^2 & \lesssim \norm{u_x}_{\infty}\norm{\rho}_{H^2}^2 + \norm{m}_{H^1}\norm{\rho}_{\infty}\norm{\rho}_{H^2} + \norm{m}_{H^1}\norm{\rho_{x}^2}_{L^2} +  \norm{u_x}_\infty\norm{\rho}_{H^2}^2 \\
\end{split}
\end{equation}
which can be further simplified using $\norm{\rho_x^2}_{L^2}\lesssim\norm{\rho}_{H^1}\norm{\rho}_{H^2}$
allowing us to write
\begin{equation}\label{s3:eq23}
\begin{split}
\frac{d}{dt}\norm{\rho_{xx}}_{L^2}^2 &\lesssim (1+\norm{m}_{L^2}+ \norm{\rho}_{H^1})(\norm{m}_{H^1}^2+\norm{\rho}_{H^2}),
\end{split}
\end{equation}
having also used $\norm{\rho}_{\infty}\lesssim \norm{\rho}_{H^1}$ and $\norm{u_x}_{\infty}\lesssim\norm{m}_{L^2}$.
It follows that
\begin{equation}\label{s2eq24}
\frac{d}{dt}(\norm{m}_{H^1}^2+\norm{\rho}_{H^2}^2)\lesssim \norm{m}_{L^2}^2+(1+\norm{\rho}_{H^1}+\norm{m}_{L^2})(\norm{m}_{H^1}^2+\norm{\rho}_{H^2}^2),
\end{equation}
thus confirming $\norm{m}_{H^1}^2 + \norm{\rho}_{H^2}^2$ remains bounded in $J$ if $\norm{u_x}_\infty$ is bounded in $J.$

To demonstrate the  general case we must show that the quantity $\norm{m}_{H^k}^2 + \norm{\rho}_{H^{k+1}}^2$
is also bounded on any bounded subinterval of $J$ for all positive integers $k$, when $\norm{u_x}_{\infty}$ is bounded. To demonstrate this we use proof by induction and so want to show if the expression is bounded for some $k>1$ then the same is true for $k+1.$ We note

\begin{equation}\label{s3:eq25}
\begin{split}
	\frac{d}{dt}(\norm{m}_{H^{k+1}}^2 + \norm{\rho}_{H^{k+2}}^2) = \frac{d}{dt}(\norm{m}_{H^{k}}^2 + \norm{\rho}_{H^{k+1}}^2)+\frac{d}{dt}\int\left( (m^{(k+1)})^2  + (\rho^{(k+2)})^2\right)
\end{split}
\end{equation}
and so it is only required that we show the integral on the right hand side remains bounded on $J$ if $\norm{u_x}_\infty$ remains bounded on $J.$

We have
\begin{multline}\label{s3:eq26}
\int\left(m_{t}^{(k+1)}m^{(k+1)} + \rho_t^{(k+2)}\rho^{(k+2)}\right) =\\
\int m^{(k+1)}\left(\alpha u_x^{(k+1)} - a (u_xm)^{(k+1)} - (um_x)^{(k+1)}-\kappa(\rho\rho_x)^{(k+1)}\right)\\
-\int\rho^{(k+2)}\left((u\rho_x)^{(k+2)}+(a-1)(u_x\rho)^{(k+2)}\right).
\end{multline}
The Leibniz rule combined with the Cauchy-Schwarz inequality ensures
\begin{multline}\label{s3:eq27}
	\abs{\int(fg)^{(n+1)}h^{(n+1)}}\lesssim \norm{f}_{H^{n}}\norm{g}_{H^{n+1}}\norm{h}_{H^{n+1}}\\
                                                                        + \abs{\int fg^{(n+1)}h^{(n+1)}} + \abs{\int f^{(n+1)}gh^{(n+1)}},
\end{multline}
which may be applied to the above expression yielding
\begin{multline}\label{s3:eq28}
\frac{d}{dt}(\norm{m}_{H^{k+1}}^2+\norm{\rho}_{H^{k+2}}^2)
\lesssim (1+\norm{m}_{H^{k}}+\norm{\rho}_{H^{k+1}})(\norm{m}_{H^{k+1}}^2+\norm{\rho}_{H^{k+2}}^2)\\
+\int_{\Circle}\alpha m^{(k+1)}u_x^{(k+1)}dx.
\end{multline}
The Cauchy-Schwarz inequality applied to the remaining integral provides the estimate
\begin{equation}\label{s3:eq29}
    \int_{\Circle}\alpha m^{(k+1)}u_x^{(k+1)}dx\lesssim \norm{m}_{H^{k+1}}\norm{u}_{H^{k+2}}\lesssim \norm{m}_{H^{k+1}}^2 + \norm{m}_{H^k}^2,
\end{equation}
having used $\norm{u}_{H^{k+2}}\leq \norm{m}_{H^k}.$
It follows at once that we may write
\begin{equation}\label{s3:eq30}
\begin{split}
\frac{d}{dt}(\norm{m}_{H^{k+1}}^2+\norm{\rho}_{H^{k+2}}^2)\lesssim\norm{m}_{H^k}^2+(1+\norm{m}_{H^k}+\norm{\rho}_{H^{k+1}})(\norm{m}_{H^{k+1}}^2+\norm{\rho}_{H^{k+2}}^2)
\end{split}
\end{equation}
and with Gronwall's Lemma this ensures the quantity $\norm{m}_{H^{k+1}}^2 + \norm{\rho}_{H^{k+2}}^2$ remains bounded, thus Lemma \ref{s3:lem2} has been verified.
\end{proof}

\subsection{Proof of Theorem \ref{thm:gc}}\label{Sec3.2}
\begin{proof}
Let $u_0,\ \rho_0\in\CS$ be given, and denote by
\begin{equation*}
    (u,\rho)\in \mathrm{C}^{\infty}(J,\CS\oplus\CS)
\end{equation*}
the unique non-extendable solution to \eqref{s1:eq1} emerging from $(u_0,\rho_0)$. Assume that $J$ is bounded and write $T^+:=\sup J$. Then there must exists a $k_0\in\mathbb{N}$ such that
\begin{equation}\label{eq:blowup}
    \limsup_{t\to T^+}\Vert u(t)\Vert_{H^{k_0}}=\infty\qquad\text{or}\qquad\limsup_{t\to T^+}\Vert \rho(t)\Vert_{H^{k_0}}=\infty.
\end{equation}
Recalling that by assumption $\kappa\ge 0$ and $a=2$, Remark \ref{rm:3} in conjunction with \eqref{eq:aprio} and \eqref{s3:eq11} implies that $\Vert u(t)\Vert_{L_\infty}$ is bounded on $J$. Invoking Lemma \ref{s3:lem2} this contradicts \eqref{eq:blowup}.
\end{proof}

\section*{Acknowledgements}
Both authors are grateful to Boris Kolev for several helpful discussions. The authors would like to thank the referees for several helpful comments.

\end{document}